\documentclass[runningheads]{llncs}

\usepackage{amsmath,amssymb,mathtools}
\usepackage{multirow}
\usepackage[shortlabels,inline]{enumitem}
\usepackage{tikz}
\usetikzlibrary{shapes.geometric, arrows}

\newcounter{cases}
\newcounter{subcases}[cases]

\newcounter{magicrownumbers}
\newcommand\rownumber{\refstepcounter{magicrownumbers}\arabic{magicrownumbers}}
\begin{document}

%\markboth{M.\ Patawar and K.\ Kapoor}{The Length of the Longest Sequence of Consecutive FS-double Squares}
\title{The Length of the Longest Sequence of Consecutive FS-double Squares in a word}% \\}
\author{Maithilee Patawar\inst{1} \and Kalpesh Kapoor\inst{2}}
\authorrunning{M.\ Patawar and K.\ Kapoor}
\institute{Department of Computer Science and Engineering\\
Indian Institute of Technology Guwahati\\
\email{maith176101104@iitg.ac.in} \and
Department of Mathematics\\
Indian Institute of Technology Guwahati\\
\email{kalpesh@iitg.ac.in}
}

\maketitle

\begin{abstract}
A square is a concatenation of two identical words, and a word $w$ is said to have a square $yy$ if $w$ can be written as $xyyz$ for some words $x$ and $z$. It is known that the ratio of the number of distinct squares in a word to its length is less than two and any location of a word could begin with at most two rightmost distinct squares. A square whose first location starts with the last occurrence of two distinct squares is an FS-double square. We explore and identify the conditions to generate a sequence of locations in a word that starts with FS-double squares. We first find the structure of the smallest word that begins with two consecutive FS-double squares and obtain its properties that enable to extend the sequence of FS-double squares. It is proved that the length of the longest sequence of consecutive FS-double squares in a word of length $n$ is at most $\frac{n}{7}$. We show that the squares in the longest sequence of consecutive FS-double squares are conjugates.
\end{abstract}

\section{Introduction}
A word is a finite sequence of symbols or letters drawn from a nonempty finite set. A repetition in a word is of the form $u^m$ for some nonempty word $u$ and an integer $m>1$. The smallest repetition obtained with $m=2$ is known as a square and is a concatenation of two identical words. The study of squares reveals interesting properties of words, and is a well-researched topic in word combinatorics, see for example \cite{Cro81,Loth05}.

Fraenkel and Simpson \cite{FS98} conjectured that the number of distinct squares in a word is always less than its length. Suppose $w=a_1 a_2 \ldots a_n$ be a word. To give an upper bound on the number of distinct squares in a word, they considered the number, $k$, of distinct squares that can start from a letter $a_i$, $1 \leq i \leq n$, for the last time in the word $w$. It is known that the value of $k$ can be at most two \cite{FS98,ilie05}.

A word that starts with two rightmost distinct squares is called as FS-double square \cite{deza15}. %Thierry \cite{thiarx20} has classified an FS-double square in a word into seven types based on its relative position and length with respect to a previous FS-double square in the word. 
We show that a sequence of consecutive locations with FS-double squares in a word is possible only under specific conditions. Our goal in this study is to elaborate on the ways that maximize consecutive FS-double squares. In particular, the following are the main contributions of our work.
\begin{enumerate}[(a)]
	\item Identify the structure of a word having FS-double squares at consecutive locations.
	\item We give an upper bound for the maximum number of consecutive FS-double squares in a word.
\end{enumerate}
The rest of the paper is organized as follows. In the next section, we introduce basic definitions and give an overview of prior work. In Section \ref{sec:structure of 2FS}, we consider all possible conditions for two consecutive FS-double squares. The feasibility of these conditions are verified against the properties of an FS-double square, and it is shown that only two structures are feasible. Further, in Section \ref{sec:longest sequence of FS} we explore the properties of two structures discovered in the previous section and employ them to obtain the length of the longest sequence of FS-double squares. Finally, we summarize in Section \ref{sec:conclusion}.

\section{Definitions and Background}\label{sec:def}
A word is a concatenation of symbols drawn from an alphabet $\varSigma$. The concatenation of two words $w_1$ and $w_2$ is represented as $w_1 \cdot w_2$ or simply $w_1 w_2$. The `$\cdot$' operation is associative. The number of letters in a word $w$ is denoted by $|w|$. The word with length zero is denoted by $\epsilon$. The $i^{th}$ power of a word $w$, denoted by $w^i$, is defined recursively as $w^0 = \epsilon$ and $w^{i}=w^{i-1}.w$ for $i \geq 1$.

Define $\varSigma^*$ ($\varSigma^+$) as set of all (nonempty) words over $\varSigma$. If $w=pqr$ then the words $p$, $q$ and $r$ are called as prefix, factor and suffix, respectively, of $w$. A proper prefix (factor, suffix) of a word, $w$, is a prefix (factor, suffix) which is not equal to $\epsilon$ or $w$. The longest common prefix of two words $w_1$ and $w_2$ is denoted as $lcp(w_1,w_2)$. A word $w'$ is said to be a conjugate of a word $w$ if and only if $w=uv$ and $w'=vu$ for some $u,v \in \varSigma^*$.

A nonempty word, $w$, is a primitive word if whenever $w=u^k$ it implies that $k=1$. A non-primitive word, $w$, satisfies the relation $w=u^k$ for some nonempty word $u$ and an integer $k \geq 2$. A square is a word, $w$, of the form $uu$ where the nonempty word $u$ is referred to as root of $uu$. If the root $u$ is a primitive word, then the square $w$ is referred to as a primitively rooted square.

The set of distinct squares in a word $w$ is denoted by $DS(w)$. Let $w = a_1a_2\ldots a_{n}$ be a word. A square $u^2$ whose last occurrence in a word begins at a location $i$ is denoted as $u_i^2$. Define $s_i$ to be the number of those squares that start at location $i$ in $w$ which do not start at location $j$ where $i < j \leq n$. It is known that for any word the value of $s_i \leq 2$ for all $1 \leq i \leq n$. As a corollary, the number of distinct squares in a word of length $n$ is bounded by $2n$ for a word defined over any finite alphabet \cite{FS98}.

As mentioned above, for a given location $i$ the value of $s_i$, where $i \in \{1, \ldots, n\}$, can be either $0, 1$ or $2$. Deza et al.\ \cite{deza15} have studied the properties of words with $s_i = 2$, that is, words in which two rightmost distinct squares start at a location. They referred to the longer square that starts at location $i$ as ``FS-double Square''. For an FS-double square that begins at location $i$, we use $sq_i^2$ and $SQ_i^2$ to denote the shorter and the longer squares, respectively, which does not appear later at any location.

\begin{lemma}[Structure of an FS-double Square \cite{deza15}]\label{lem:2FS structure}
	The roots, $sq_i$ and $SQ_i$, of an FS-double square starting at location $i$ have the following structure:
	\begin{align}
	sq_i&=(x_1x_2)^{p_1}(x_1)\\
	SQ_i&=(x_1x_2)^{p_1}(x_1)(x_1x_2)^{p_2}
	\end{align}
	for some integer $p_1,p_2$ such that $p_1\geq p_2\geq 0$ and the non-empty words $x_1,x_2$ result into a primitive word $x_1x_2$.
\end{lemma}
It is shown in \cite{deza15} that a word of length $n$ can have at most $\lfloor{\frac{5n}{6}}\rfloor$ distinct FS-double squares. This result, in turn, gives a bound of $\lfloor{\frac{11n}{6}}\rfloor$ on the number of distinct squares. The earlier work \cite{deza15} has focused on the characteristics of a pair of FS-double squares separated by one or more locations. In this paper, we restrict the study to FS-double squares that appear at consecutive locations. This has helped in the exact characterization of the structure of words with the longest chain of FS-double squares.

\subsection{Mates of an FS-double Square}\label{subsec:mates}
For an FS-double square $SQ_1^2$, Deza et al. \cite{deza15} categorized another FS-double square $SQ_k^2,k>1$ into five types based on the value of $k$ and the sizes of roots $\{|sq_1|,|SQ_1|\},\{|sq_k|,|SQ_k|\}$ . For $k<(p_1-1)|x_1x_2|+|lcp(x_1x_2,x_2x_1)|$, an FS-double square $SQ_k^2$ is one of the following mates of $SQ_1^2$ if it satisfies the respective conditions \cite{deza15}.
\begin{enumerate}[(a)]
	\item $\alpha$ mate: The longer root $SQ_i$ is conjugate of $SQ_1$ which gives $|SQ_1|=|SQ_k|$.
	\item $\beta$ mate: The root $sq_k=(\Tilde{x})^i(\Tilde{x_1})$ where $1<i<p_1$ and $\Tilde{x}$ respective $\Tilde{x_1}$ is conjugate of $x_1x_2$ respectively $x_1$. This implies $|sq_1|<|sq_k|$, $|SQ_1|=|SQ_k|$.
	\item $\gamma$ mate: Here, $k<p_1|x_1x_2|$ and  $|sq_k|=|SQ_1|$.
	\item $\delta$ mate: The lengths of the roots satisfy $|sq_k|>|SQ_1|$ (The complete definition of $\delta$ mates given in \cite{deza15} is discussed later in Definition \ref{def:delta}).
\end{enumerate}
An $\epsilon$ mate of $SQ_1^2$ starts after $(p_1-1)|x_1x_2|+|lcp(x_1x_2,x_2x_1)|$ locations. Since we are interested in consecutive FS-double squares, we omit to discuss this mate here. The two FS-double squares beginning at adjacent locations are definitely one of the above four mates. It is easier to show that these FS-double squares can be $\alpha$ mates compare to other possible mates. However, we need more details of FS-double squares starting at consecutive locations to check them against the definitions of $\beta$ , $\gamma$ and $\delta$ mates. In the next section, we show that these FS-double squares are either $\alpha$ or $\delta$ mates and cannot be $\beta$ or $\gamma$ mates.

\section{Structure of 2FS Squares}\label{sec:structure of 2FS}
We refer to a word starting at location $i$ with $s_i = s_{i+1} = 2$ as a 2FS square. Let $w = a_1 \ldots a_{n}$ be a 2FS square in which $s_1 = s_2 = 2$ with $(sq_1, SQ_1)$ and $(sq_2, SQ_2)$ be the two respective pairs of roots of FS-double squares. As a convention, for an FS-double square, we assume that $|sq_1| < |SQ_1|$ and $|sq_2| < |SQ_2|$.

To identify the structure of a 2FS square, we consider all possible relations between the lengths of $sq_1$, $sq_2$, $SQ_1$ and $SQ_2$. We show that, in some 2FS squares, the two consecutive FS-double squares are conjugates.

A square $SQ^2$ is labelled as a balanced double square if one of its prefix is a square, say $sq^2$, and the roots satisfy the relation $|sq| < |SQ| < 2|sq|$ \cite{deza15}. Similar to an FS-double square, a balanced double square starts with two nonempty squares $(sq^2,SQ^2)$, but their occurrence in a word is not necessarily the last. A canonical factorization of a balanced double square is given in \cite{deza15}. We use the following existing results of balanced double squares to characterize 2FS squares.

\begin{lemma}[Balanced Double Square \cite{deza15}]\label{balancedDB}
	A square, $SQ^2$, that starts with another shorter square, $sq^2$, is a balanced double square if one of the following conditions holds:
	\begin{enumerate}[(a)]
		\item either $sq$ or $SQ$ is primitively rooted, or
		\item \label{conb} prefix $sq^2$ is the last occurrence in $SQ^2$.
	\end{enumerate}
\end{lemma}

An FS-double square is a balanced double square since it satisfies the constraint (\ref{conb}) of Lemma \ref{balancedDB}. Therefore, the following \textit{two squares factorization lemma} for balanced double squares is applicable for FS-double squares as well.

\begin{lemma}[Two Squares Factorization Lemma \cite{bai16NPL}]\label{lem:primitiveBases}
	Given an FS-double square, $SQ_i^2$, that begins with a shorter square, $sq_i^2$, such that $sq_i=(x_1x_2)^{p_1}x_1$ and $SQ_i=sq_i(x_1x_2)^{p_2}$ for some nonempty words $x_1,x_2$ and integers $p_1\geq p_2\geq 0$ then,
	\begin{enumerate}[(a)]
		\item $SQ_i$ is primitively rooted, and
		\item $sq_i$ is primitively rooted if $p_1>1$.
	\end{enumerate}
\end{lemma}
The following Lemma gives a way to introduce distinct squares of equal lengths that start at consecutive locations by extending a primitively rooted square. Further, the obtained squares are conjugates of each other.
\begin{lemma}[\cite{mai21}]\label{lem:appendingPrefix}
	Let $w=uu$ be a primitively rooted square such that $|u|>1$ and $v$ be a proper prefix of $u$. Then, the suffix $v$ in the word $wv$ introduces $|v|$ distinct conjugates of $uu$.
\end{lemma}
It is also possible to extend a non-primitive square to obtain new squares. However, the newly introduced squares may not be the distinct squares. In general, the new squares introduced in both the cases are conjugates of the given square.
\begin{lemma}\label{lem:EqualLenConseSq}
	Let $uu$ and $vv$ be two squares with $|u|=|v|$ that appear at consecutive locations in a word. Then, the pairs ($u$, $v$) and ($uu$, $vv$) are conjugates. Further, $ua = av$ and $uua = avv$ for some $a \in \varSigma$.
\end{lemma}
\begin{proof}
	Let $uu$ and $vv$ be two consecutive squares in a word where the former square ends first. Assume $u$ begins with a letter $a$ and $v$ ends with a letter $b$ such that $u=au',v=v'b$. The squares $uu$ and $vv$ appear at consecutive locations leading to a relation $uub=avv$. So, the highlighted factors in $a\mathbf{u'a}u'b=uub$ and $a\mathbf{v'b}v'b=avv$ must be equal. Therefore, $u'=v'$, $a=b$, $uu=av'av'$ and $vv=v'av'a$. Thus, the squares $uu$ and $vv$ are conjugates.
\end{proof} In Subsection \ref{subsec:mates}, we discussed the definitions of different `mates' that are given in \cite{deza15}. Accordingly, an $\alpha$ mate of an FS-double square $SQ_1^2$ refers to an FS-double square $SQ_k^2$ where $k>1$ and $|sq_1|=|sq_k|,|SQ_1|=|SQ_k|$. 
In the next subsection we show that for $k=2$, $|SQ_1|=|SQ_2|\iff |sq_1|=|sq_2|$.

\subsection{Equal 2FS Squares}\label{subsec:Equal2FS}
A 2FS square starts with two FS-double squares, say $SQ_1^2$ and $SQ_2^2$, where
$$
\begin{array}{ll}
SQ_1 = & (x_1x_2)^{p_1}(x_1)(x_1x_2)^{p_2}\\
SQ_2 = & (y_1y_2)^{q_1}(y_1)(y_1y_2)^{q_2}
\end{array}
$$
In the rest of the paper, we assume that $SQ_1$ begins with a letter $a$ such that $x_1=ax_1'$.
\begin{lemma}\label{lem:SQ0=SQ1}
	Let $w$ be a 2FS square with  $|SQ_1|=|SQ_{2}|$. Then, $|sq_1|=|sq_{2}|$.
\end{lemma}
\begin{proof}
	The longer root $SQ_1$ begins with a shorter root $sq_1=(x_1x_2)^{p_1}(x_1)$ as shown below.
	\begin{align}
	SQ_1 &= (x_1x_2)^{p_1}x_1(x_1x_2)^{p_2}=(ax_1'x_2)^{p_1}ax_1'(ax_1'x_2)^{p_2} \label{Eq:sq1}
	\end{align}
	Given $|SQ_1|=|SQ_{2}|$, the roots $SQ_1$ and $SQ_{2}$ are conjugates (see Lemma \ref{lem:EqualLenConseSq}). Thus, $SQ_{2} = (x_1'x_2a)^{p_1}x_1'a(x_1'x_2a)^{p_2}$.
	We use Lemma \ref{lem:2FS structure} to get
	$sq_{2}=(x_1'x_2a)^{p_1}x_1'a$. From Equation (\ref{Eq:sq1}), the root $sq_1$ is $(ax_1'x_2)^{p_1}ax_1'$. Hence $|sq_1|=|sq_{2}|$.
\end{proof}
\begin{lemma}\label{lem:sq0=sq1,SQ0<SQ1}
	Let $w$ be a 2FS square in which $|sq_1|=|sq_{2}|$. Then, $|SQ_1|= |SQ_{2}|$.
\end{lemma}
\begin{proof}We prove the claim by contradiction by showing that both $|SQ_1| < |SQ_2|$ and $|SQ_1| > |SQ_2|$ are not possible.
	
	Suppose $|SQ_1|<|SQ_{2}|$. Assume $SQ_1=(x_1x_2)^{p_1}(x_1)(x_1x_2)^{p_2}$ begins with a letter $a$ such that $x_1x_2=ax_1'x_2$ and $sq_1=(ax_1'x_2)^{p_1}(ax_1')$. We use Lemma \ref{lem:EqualLenConseSq} to get $sq_2 = (x_1'x_2a)^{p_1}(x_1'a)$. As $|SQ_1|<|SQ_2|$, it must be the case that
	\begin{align}
	SQ_2 =& (x_1'x_2a)^{p_1}(x_1'a)(x_1'x_2a)^{p_2+p_3} \text{ where } p_3 > 0
	\end{align}
	Using the above equation, we can express the squares $SQ_1^2$ and $SQ_2^2$ as follows,
	\begin{align}
	SQ_1^2 =& a(x_1'x_2a)^{p_1}(x_1'a)(x_1'x_2a)^{p_2}\underline{(x_1'x_2a)^{p_1}(x_1'a)(x_1'x_2a)^{p_2-1}x_1'x_2}\label{eq:sqCompare1}\\
	SQ_2^2 =&~~(x_1'x_2a)^{p_1}(x_1'a)(x_1'x_2a)^{p_2}\underline{(x_1'x_2a)^{p_3}(x_1'x_2a)^{p_1}(x_1'a)(x_1'}x_2a)^{p_2+p_3}\label{eq:sqCompare2}
	\end{align}
	Note that the squares, $SQ_1$ and $SQ_2$, appear at consecutive locations. It can be observed by comparing the underlined factors in Equations (\ref{eq:sqCompare1}) and (\ref{eq:sqCompare2}) that the word $x_2$ begins with the letter $a$, and
	\begin{align*}
	(x_1'a)(x_1'x_2a)&=x_1'x_2ax_1'a\\
	ax_1'x_2&=x_2ax_1' \implies x_1x_2=x_2x_1
	\end{align*}
	The relation $x_1x_2=x_2x_1$ is a contradiction since $x1x2$ is primitive.
	
	Let us now consider the condition $|sq_1|=|sq_{2}|$ and $|SQ_1|>|SQ_{2}|$. The lengths of suffixes of $SQ_1$ and $SQ_2$ in this 2FS square are $p_2|x_1x_2|$ and $q_2|y_1y_2|$, respectively. Note that $p_2|x_1x_2| < q_2|y_1y_2|$. We can rewrite the FS-double squares using the given conditions as follows.
	\begin{align}\label{caseC squares}
	SQ_1^2 =&a(x_1'x_2a)^{p_1}(x_1'a)&&\mathbf{(x_1'x_2a)^{p_2}}(x_1'x_2a)^{p_1}(x_1'a)(x_1'x_2a)^{p_2-1}(x_1'a)\\
	SQ_2^2 =&~(y_1y_2)^{q_1}\;\;\;\;(y_1)&&\mathbf{(y_1y_2)^{q_2}}(y_1y_2)^{q_1}(y_1)(y_1y_2)^{q_2}
	\end{align}
	The highlighted factors given in Equations (\ref{caseC squares}) lead to, $$(x_1'x_2a)^{p_2}=(y_1y_2)^{q_2}.k \textnormal{\hspace*{15pt} where $k$ is some nonempty word}$$
	Substituting this in Equation (\ref{caseC squares}) gives,
	\begin{align}
	SQ_1^2 =&a(x_1'x_2a)^{p_1}(x_1'a)&&(x_1'x_2a)^{p_2}(x_1'x_2a)^{p_1}(x_1'a)(x_1'x_2a)^{p_2-1}(x_1'a)\nonumber\\
	SQ_2^2 =&~(y_1y_2)^{q_1}\;\;\;\;(y_1)&&\mathbf{(y_1y_2)^{q_2}k.(y_1y_2)^{q_1}}(y_1)(y_1y_2)^{q_2}k\label{eq:9}\\
	SQ_2^2 =&~(y_1y_2)^{q_1}\;\;\;\;(y_1)&&\mathbf{(y_1y_2)^{q_2}(y_1y_2)^{q_1}(y_1)}(y_1y_2)^{q_2}\label{eq:10}
	\end{align}
	In following cases, the highlighted part of $SQ_2^2$ given in Eq. \ref{eq:9} and \ref{eq:10} is discussed for all possible $k$. Each of these cases leads to a contradiction that $y_1y_2$ is non-primitive.
	\begin{description}
		\item [Case 1] $k=(y_1y_2)^{q_3}$: Here $0<q_3<(q_1-q_2)$ as $0<q_2\leq q_1$.
		\begin{align*}
		(y_1y_2)^{q_2}k(y_1y_2)^{q_1}(y_1) &=(y_1y_2)^{q_2}(y_1y_2)^{q_1}(y_1)(y_1y_2)^{q_2}\\
		(y_1y_2)^{q_3}(y_1y_2)^{q_1}(y_1) &=(y_1y_2)^{q_1}(y_1)(y_1y_2)^{q_2}\implies
		y_1y_2=y_2y_1
		\end{align*}
		\item [Case 2] $k=y_1$: We get $y_1.(y_1y_2)^{q_1}=(y_1y_2)^{q_1}.y_1 \implies y_1y_2=y_2y_1$
		\item [Case 3] $k=(y_1y_2)^{q_3}y_3$ where $y_1y_2=y_3y_4$
		\begin{align*}
		k(y_1y_2)^{q_1}&=(y_1y_2)^{q_1}(y_1)(y_1y_2)^{q_2}\\\
		{(y_1y_2)^{q_3}}y_3(y_1y_2)^{q_1}&={(y_1y_2)^{q_3}}(y_1y_2)^{q_1-q_3}(y_1)(y_1y_2)^{q_2}\implies y_3y_4 = y_4y_3
		\end{align*}
	\end{description}
	Thus, $|SQ_1|$ cannot be more than $|SQ_2|$ if the respective shorter squares of FS-double squares are of the same lengths.
\end{proof}

\begin{theorem} [Equal 2FS Squares]\label{thrm:equal2FS}
	In a 2FS square, $|SQ_1|=|SQ_2|$ if and only if $|sq_1|=|sq_2|$.
\end{theorem}
\begin{proof} Follows from Lemmas \ref{lem:SQ0=SQ1} and \ref{lem:sq0=sq1,SQ0<SQ1}.
\end{proof}
\begin{corollary}\label{cor:lenOfEqual2FS}
	Let $w$ be a 2FS square. Then, the words $SQ_1^2$ and $SQ_2^2$ are conjugates if and only if $sq_1^2$ and $sq_2^2$ are conjugates.
\end{corollary}
\begin{proof}
	The statement follows from Theorem \ref{thrm:equal2FS} and the fact that conjugates are of equal length.
\end{proof}
Let us see the condition that is required to have an equal 2FS square. Suppose $w$ is an FS-double square with $sq_1^2$ and $SQ_1^2$. It is possible to extend $w$ to obtain a 2FS square. According to Lemmas \ref{lem:EqualLenConseSq} and \ref{lem:sq0=sq1,SQ0<SQ1}, the shorter squares among the two consecutive FS-double squares in an equal 2FS square are conjugates. Therefore, the word $w$ must have a square at the second location, which must be a conjugate of $sq_1^2$.
\begin{align*}
SQ_1^2&=(ax_1'x_2)^{p_1}.(ax_1').(ax_1'x_2)^{p_2}.(ax_1'x_2)^{p_1}.(ax_1').(ax_1'x_2)^{p_2}\\
&=\underline{(ax_1'x_2)^{p_1}.(ax_1').(ax_1'x_2)^{p_1}.(ax_1'}x_2)^{p_2}.(ax_1').(ax_1'x_2)^{p_2}\\
&=(a\underbrace{x_1'x_2)^{p_1}.(ax_1').(ax_1'x_2)^{p_1}.ax_1'(x_2}_{sq_2^2}ax_1')^{p_2}.(ax_1'x_2)^{p_2}
\end{align*}
Thus, the necessary condition for an equal 2FS square is that $lcp(x_1,x_2)$ must be nonempty. In the following lemma, we give the maximum number of consecutive FS-double squares that can be conjugates of each other.
\begin{lemma}\label{lem: max equal 2FS}
	Let a word $w$ that begins with $i$ consecutive FS-double squares such that $|SQ_1|=|SQ_2|=\cdots =|SQ_i|$
	% and $|sq_1|=|sq_2|=\cdots =|sq_i|$
	where $i$ is an positive integer. Then,
	$$
	i \leq \left\{
	\begin{array}{ll}
	|lcp(x_1x_2,x_2x_1)|+1 & \textnormal{ if } p_1>p_2\\
	min(|lcp(x_1x_2,x_2x_1)|+1,|x_1|) & \textnormal{ otherwise}
	\end{array}
	\right.
	$$
\end{lemma}
\begin{proof}
	As the consecutive FS-double squares are of equal lengths, they are conjugates (refer Theorem \ref{thrm:equal2FS}). Similarly, the shorter squares in these FS-double squares are also conjugates. As per Lemma \ref{lem:appendingPrefix}, conjugates of $sq_1^2$ at consecutive locations are possible if $sq_1^2$ is extended with it's prefix.
	\begin{align*}
	SQ_1^2=&(x_1x_2)^{p_1}(x_1)(x_1x_2)^{p_2}.(x_1x_2)^{p_1}(x_1)(x_1x_2)^{p_2}\\
	=&\underbrace{(\underline{x_1x_2})^{p_1}(x_1)(x_1x_2)^{p_1}x_1}_{sq_1^2}.(\underline{x_2x_1})^{p_2}(x_1x_2)^{p_2}
	\end{align*}
	The word after $sq_1^2$ in $SQ_1^2$ that matches with the prefix of $sq_1^2$ decides the number of conjugates. Since $x_1x_2\neq x_2x_1$, the number of conjugates, in this case, is $|lcp(x_1x_2,x_2x_1)|$. We can get at most $|u|-1$ conjugates of a square $uu$ by appending it's prefix such that these conjugates start at consecutive locations. For $sq_1^2$, there can be $|sq_1|-1$ conjugates and $|lcp(x_1x_2,x_2x_1)|<|sq_1|-1$. So, the total number of FS-double squares are $|lcp(x_1x_2,x_2x_1)|+1$. We know that $p_1\geq p_2$ (see Lemma \ref{lem:2FS structure}) and this value of $i$ holds for $p_1>p_2$. 
	
	It is also necessary to take care of earlier FS-double squares while extending the larger square $SQ_1^2$. For $p_1=p_2=p$, the square $sq_1^2$ repeats if $SQ_1^2$ extended with its prefix of length $|x_1|$ or more.
	\begin{align*}
	SQ_1^2=&(x_1x_2)^{p_1}(x_1)(x_1x_2)^{p_2}.(x_1x_2)^{p_1}(x_1)(x_1x_2)^{p_2}\\
	SQ_1^2.(x_1)=&(x_1x_2)^{p}(x_1)(x_1x_2)^{p}.\boxed{(x_1x_2)^{p}(x_1)(x_1x_2)^{p}(x_1)}\implies{sq_1^2}
	\end{align*}
	Thus, the first location of $SQ_1^2.x_1$ starts with only one rightmost square. So, the maximum value of $i$ is $min(|lcp(x_1x_2,x_2x_1)|+1,|x_1|)$.
\end{proof}
\subsection{Unequal 2FS Squares}\label{subsec:Unqual2FS}
In the previous section, we proved the structure of a 2FS squares where the consecutive FS-double squares satisfy $|SQ_1|=|SQ_2|\iff |sq_1|=|sq_2|$. Referring in terms of the mates defined in \cite{deza15}, an FS-double square $SQ_2^2$ in a 2FS square can be an $\alpha$ mate of the first FS-double square. According to the definition of mates, the second FS-double square in a 2FS square
can be a $\beta,\gamma$ or a $\delta$ mate if the lengths of squares satisfy the respective definitions (see definitions in \cite{deza15}). In this section, we show that $SQ_2^2$ cannot be either a $\beta$ or a $\gamma$ mate of $SQ_1^2$.
\begin{lemma}\label{lem:s0>s1,S0>S1 and s0>s1,S0<S1}
	Given a 2FS square such that $|SQ_1|\neq |SQ_2|$. Then, $|sq_1|<|sq_2|$.
\end{lemma}
\begin{proof}(\textit{By contradiction})
	Assume a 2FS square with $|SQ_1|\neq |SQ_2|$ and it also satisfies $|sq_1|>|sq_2|$. Note that $|sq_1|$ cannot be equal to $|sq_2|$ (refer Theorem \ref{thrm:equal2FS}). Consider the following two cases.
	\begin{description}
		\item [Case1 $|SQ_1|>|SQ_2|$:] The two squares $SQ_1^2$ and $SQ_2^2$ are,
		\begin{align}\label{eq:shorter sq1}
		SQ_1^2=&a(x_1'x_2a)^{p_1}(x_1'a)(x_1'x_2a)^{p_2}(x_1'x_2a)^{p_1}(t'a)(x_1'x_2a)^{p_2-1}(x_1'x_2)\\
		SQ_2^2=&~(y_1y_2)^{q_1}(y_1)(y_1y_2)^{q_2}.(y_1y_2)^{q_1}(y_1)(y_1y_2)^{q_2}
		\end{align}
		Figure \ref{fig:cmpSQ} shows the overlap between these longer squares after removing the first letter $a$ from $SQ_1$.
	\end{description}
	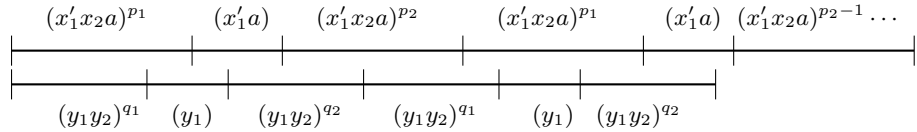
\begin{figure}[htb]
		\begin{tikzpicture}[xscale=1.2]
		\draw [thick]  (0,0) -- (10,0);
		\draw (0,-.2) -- (0, .2);
		\draw (2,-.2) -- (2, .2);
		\draw (3,-.2) -- (3, .2);
		\draw (5,-.2) -- (5, .2);
		\draw (7,-.2) -- (7, .2);
		\draw (8,-.2) -- (8, .2);
		\draw (10,-.2) -- (10, .2);
		\node[align=left, above] at (0.95,.15){$(x_1'x_2a)^{p_1}$};
		\node[align=center, above] at (2.55,.15){$(x_1'a)$};
		\node[align=left, above] at (3.95,.15){$(x_1'x_2a)^{p_2}$};
		\node[align=left, above] at (5.95,.15){$(x_1'x_2a)^{p_1}$};
		\node[align=center, above] at (7.55,.15){$(x_1'a)$};
		\node[align=left, above] at (8.95,.15){$(x_1'x_2a)^{p_2-1}\cdots$};
		\end{tikzpicture}
		
		\begin{tikzpicture}[xscale=1.2]
		\draw [thick]  (0,0) -- (7.8,0);
		\draw (0,-.2) -- (0, .2);
		\draw (1.5,-.2) -- (1.5, .2);
		\draw (2.4,-.2) -- (2.4, .2);
		\draw (3.9,-.2) -- (3.9, .2);
		\draw (5.4,-.2) -- (5.4, .2);
		\draw (6.3,-.2) -- (6.3, .2);
		\draw (7.8,-.2) -- (7.8, .2);
		\node[align=right, below] at (0.80,-.15){$~~~~(y_1y_2)^{q_1}$};
		\node[align=center, below] at (2,-.15){$(y_1)$};
		\node[align=left, below] at (3.2,-.15){$(y_1y_2)^{q_2}$};
		\node[align=left, below] at (4.7,-.15){$(y_1y_2)^{q_1}$};
		\node[align=center, below] at (6,-.15){$(y_1)$};
		\node[align=left, below] at (6.95,-.15){$(y_1y_2)^{q_2}$};
		\end{tikzpicture}
		\caption{Aligned bigger squares in a 2FS square with $|sq_1| > |sq_2|$}
		\label{fig:cmpSQ}
	\end{figure}
	From the relation $|sq_2^2|<|sq_1^2|$ and Equation (\ref{eq:shorter sq1}), the word $(y_1y_2)^{q_1}$ is a prefix of $(x_1'x_2a)^{p_1}$. The occurrence of $(y_1y_2)^{q_1}$ in $SQ_1^2$ is shown below:
	\begin{align*}
	SQ_1^2=a(x_1'x_2a)^{p_1}(x_1'\overbrace{a)\underbrace{(x_1'x_2a)^{p_2}(x_1'}_{(y_1y_2)^{q_2}}}^{(y_1y_2)^{q_1+q_2}}a)^{p_1}(x_1'a)(x_1'x_2a)^{p_2-1}(x_1'x_2)
	\end{align*}
	It can be noted that $y_1y_2$ is a non-primitive word. According to Lemma \ref{lem:primitiveBases}, if the second location starts with an FS-double square, then $y_1y_2$ must be a primitive word. This is a contradiction.
	\begin{description}
		\item [Case2 $|SQ_1|<|SQ_2|$:] Similar to the previous case, we write $SQ_1^2$ as follows.
		\begin{align*}
		SQ_1^2=a(x_1'x_2a)^{p_1}(x_1'\overbrace{a)\underbrace{(x_1'x_2a)^{p_2}.(x_1'}_{y_1y_2^{q_2}}x_2a)^{p_1}(x_1'}^{(y_1y_2)^{q_1+q_2}}a)(x_1'x_2a)^{p_2-1}(x_1'x_2)
		\end{align*}Again, it shows that $y_1y_2$ is a non-primitive word which is a contradiction.
	\end{description}
	Thus, we have shown that the longer roots of a 2FS square if satisfy the relation $|SQ_1|\neq|SQ_2|$, then the lengths of the shorter roots follow $|sq_1|<|sq_2|$.
\end{proof}
The previous result rules out the possibility of $|sq_1| > |sq_2|$ for smaller roots in a 2FS square. The following lemma shows that the second FS-double square in any 2FS square is at least as long as the first one.
\begin{lemma}\label{lem:s1<s2implyS1<S2}
	Let $w$ be a 2FS square. Then, the square $SQ_2^2$ ends after $SQ_1^2$.
\end{lemma}
\begin{proof}
	If $w$ is an equal 2FS square, then $SQ_2^2$ always ends after $SQ_1^2$ (see Theorem \ref{thrm:equal2FS}).
	Now consider if $w$ is not an equal 2FS square. For this case, it is shown in Lemma \ref{lem:s0>s1,S0>S1 and s0>s1,S0<S1} that $|sq_1|<|sq_2|$. Suppose, for contradiction, $SQ_1^2$ ends after $SQ_2^2$. The two FS-double squares in such a word are given as follows:
	\begin{align}
	SQ_1^2=&a\mathbf{(x_1'x_2a)^{p_1}(x_1'a)\;\;(x_1'x_2a)^{p_2}}(x_1'x_2a)^{p_1}(x_1'a)(x_1'x_2a)^{p_2-1}(x_1'a) \label{eq:lem:s11}\\
	SQ_2^2=&~\mathbf{(y_1y_2)^{q_1}\;\;\;\;\;\;\;\;(y_1)(y_1y_2)^{q_2}}(y_1y_2)^{q_1}(y_1)(y_1y_2)^{q_2} \label{eq:lem:s12}
	\end{align}
	From the highlighted factor in Equation (\ref{caseC squares}), we have
	\[(x_1'x_2a)^{p_1}(x_1'a)(x_1'x_2a)^{p_2}=(y_1y_2)^{q_1}(y_1)(y_1y_2)^{q_2}.k  \textnormal{\hspace*{10pt} where } k\in \varSigma^+. \] Substituting this in Equations (\ref{eq:lem:s11}) and (\ref{eq:lem:s12}),
	\begin{align}
	SQ_1^2=&a(x_1'x_2a)^{p_1}(x_1'a)\;\;\;(x_1'x_2a)^{p_2}(x_1'x_2a)^{p_1}(x_1'a)(x_1'x_2a)^{p_2-1}(x_1'a)\\
	SQ_2^2=&~(y_1y_2)^{q_1}\;\;\;\;\;\;\;\;(y_1)\mathbf{(y_1y_2)^{q_2}k.(y_1y_2)^{q_1}}(y_1)(y_1y_2)^{q_2}k\\
	SQ_2^2=&~(y_1y_2)^{q_1}\;\;\;\;\;\;(y_1)\mathbf{(y_1y_2)^{q_2}(y_1y_2)^{q_1}(y_1)}(y_1y_2)^{q_2}
	\end{align}
	The highlighted factors in the above equations lead to the cases that are discussed in Lemma \ref{lem:sq0=sq1,SQ0<SQ1}. Therefore, we have a contradiction as the first location cannot start with an FS-double square.
\end{proof}
If the lengths of shorter squares of two FS squares in a 2FS square are equal, then the respective longer squares are conjugates. The only possible structure of a 2FS square where the shorter squares are of unequal lengths satisfies $|sq_1|<|sq_2|$ and $|SQ_1|<|SQ_2|$. We now find the smallest length of $sq_2^2$ in the following lemma.
\begin{lemma}\label{lem:s0<s1<=s0+S0,S0<S1}
	Let $w$ be a 2FS square with $|sq_2|>|sq_{1}|$. Then, $|sq_{2}|>|SQ_1|$.
\end{lemma}
\begin{proof}(\textit{By contradiction}) Suppose the shorter square, $sq_2^2$, in a 2FS square satisfies the relation $|sq_{2}|\leq|SQ_1|$. Then, the two possible cases to consider are as follows.
	\begin{description}
		\item [Case 1 $|sq_1|<|sq_2|<|SQ_1|<|SQ_2|$:] 
		The following equations show the alignment of two FS-double squares.
		\begin{align}
		SQ_1^2=&a(x_1'x_2a)^{p_1}(x_1'a)(x_1'x_2a)^{p_2}.\overbrace{(x_1'x_2a)^{p_1}(x_1'a)(x_1'}^{(y_1y_2)^{q_1}}x_2a)^{p_2-1}x_1'x_2 \label{eq:3.15} \\
		SQ_{2}^2=&~~(y_1y_2)^{q_1}\;\;\;\;\;\;\;\;\;\;(y_1)\;\;\;\mathbf{(y_1y_2)^{q_2}.(y_1y_2)^{q_1}}\;\;\;\;\;\;\;\;\;\;(y_1)\;\;\;(y_1y_2)^{q_2}\label{eq:3.16}
		\end{align}
		From the constraints on the roots, $|sq_1|<|sq_2|$ and $|SQ_1|<|SQ_2|$, the marked factor $(y_1y_2)^{q_1}$ in Equation (\ref{eq:3.15}) always begins somewhere in the suffix $(y_1y_2)^{q_2}$ of the first occurrence of $SQ_2$. Otherwise, the constraint $|SQ_1|<|SQ_2|$ will be violated. Next, Figure \ref{fig:2} illustrates the occurrences of $sq_2$ in $SQ_2^2$.
	\end{description}
	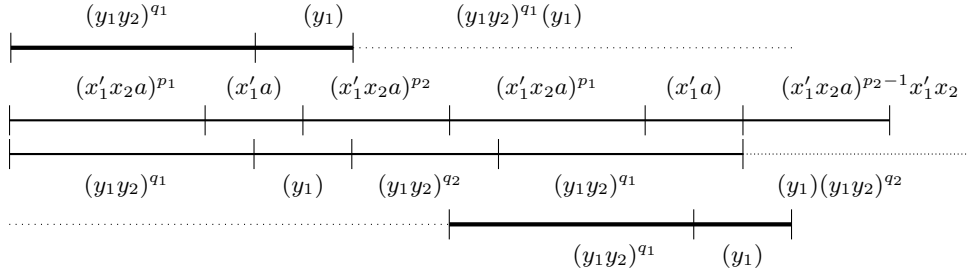
\begin{figure}[!htb]
		\begin{tikzpicture}[xscale=1.3]
		\draw [ultra thick]  (0,0) -- (3.5,0);
		\draw (0,-.2) -- (0, .2);
		\draw (2.5,-.2) -- (2.5, .2);
		\draw (3.5,-.2) -- (3.5, .2);
		\draw [dotted] (3.5,0) -- (8,0);
		\node[align=right, above] at (1.2,.15){$(y_1y_2)^{q_1}$};
		\node[align=left, above] at (3.2,.15){$(y_1)$};
		\node[align=right, above] at (5.2,.15){$(y_1y_2)^{q_1}(y_1)$};
		\end{tikzpicture}
		
		\begin{tikzpicture}[xscale=1.3]
		\draw [thick]  (0,0) -- (9,0);
		\draw (0,-.2) -- (0, .2);
		\draw (2,-.2) -- (2, .2);
		\draw (3,-.2) -- (3, .2);
		\draw (4.5,-.2) -- (4.5, .2);
		\draw (6.5,-.2) -- (6.5, .2);
		\draw (7.5,-.2) -- (7.5, .2);
		\draw (9,-.2) -- (9, .2);
		\node[align=right, above] at (1.2,.15){$(x_1'x_2a)^{p_1}$};
		\node[align=right, above] at (2.5,.15){$(x_1'a)$};
		\node[align=right, above] at (3.8,.15){$(x_1'x_2a)^{p_2}$};
		\node[align=right, above] at (5.5,.15){$(x_1'x_2a)^{p_1}$};
		\node[align=right, above] at (7,.15){$(x_1'a)$};
		\node[align=right, above] at (8.8,.15){$(x_1'x_2a)^{p_2-1}x_1'x_2$};
		\end{tikzpicture}
		
		\begin{tikzpicture}[xscale=1.3]
		\draw [thick]  (0,0) -- (7.5,0);
		\draw (0,-.2) -- (0, .2);
		\draw (2.5,-.2) -- (2.5, .2);
		\draw (3.5,-.2) -- (3.5, .2);
		\draw (5,-.2) -- (5, .2);
		\draw (7.5,-.2) -- (7.5, .2);
		\draw [densely dotted]  (7.5,0) -- (9.8,0);
		\node[align=right, below] at (1.2,-.15){$(y_1y_2)^{q_1}$};
		\node[align=right, below] at (3,-.15){$(y_1)$};
		\node[align=right, below] at (4.2,-.15){$(y_1y_2)^{q_2}$};
		\node[align=right, below] at (6,-.15){$(y_1y_2)^{q_1}$};
		\node[align=right, below] at (8.5,-.15){$(y_1)(y_1y_2)^{q_2}$};
		\end{tikzpicture}
		
		\begin{tikzpicture}[xscale=1.3]
		\draw [dotted]  (0,0) -- (4.5,0);
		\draw [ultra thick]  (4.5,0) -- (8,0);
		\draw (4.5,-.2) -- (4.5, .2);
		\draw (7,-.2) -- (7, .2);
		\draw (8,-.2) -- (8, .2);
		\node[align=right, below] at (6.2,-.15){$(y_1y_2)^{q_1}$};
		\node[align=left, below] at (7.5,-.15){$(y_1)$};
		\end{tikzpicture}
		\caption{Occurrences of $sq_2$ in $SQ_2^2$}
		\label{fig:2}
	\end{figure}
	We can conclude from Figure \ref{fig:2} that $(y_1y_2)$ appears in a square $(y_1y_2)^2$. So, $(y_1y_2)$ is a non-primitive word and the second location cannot start with an FS-double square. Thus, a contradiction.
	\begin{description}
		\item [Case 2 $|sq_1|<|sq_2|=|SQ_1|<|SQ_2|$:] The square $SQ_1^2$ begins with a letter $a$ and it must be appended by a letter $b$ to avoid the case that results in an equal 2FS square. As per the given conditions, the word $a.sq_2^2$ is the prefix of given 2FS square.  We illustrate the occurrences of the roots of squares of given 2FS square in the next equation.
		\begin{align}
		SQ_1^2=a&(x_1'x_2a)^{p_1}(x_1'a)(x_1'x_2a)^{p_2}(x_1'x_2a)^{p_1}(x_1'a)(x_1'x_2a)^{p_2-1}(x_1'x_2)\label{eq:case2.0}\\
		a.SQ_2^2=a&\overbrace{(x_1'x_2a)^{p_1}(x_1'a)(x_1'x_2a)^{p_2}}^{sq_2}\overbrace{(x_1'x_2a)^{p_1}(x_1'a)(x_1'x_2a)^{p_2-1}(x_1'x_2\mathbf{b})}^{sq_2}\cdots\label{eq:case2.1}
		\end{align}
		In Equation \ref{eq:case2.1}, the first occurrence of the root $sq_2$ in $sq_2^2$ ends with a letter $a$. We get $b=a$ from the last occurrence of $sq_2$. Hence, a contradiction.
	\end{description}
\end{proof}

\begin{theorem}[Unequal 2FS Square]\label{thrm:unequal2FS}
	For a 2FS square, $|SQ_1| < |sq_2|$ if and only if $|sq_1| < |sq_2|$.  
\end{theorem}
\begin{proof}
	We know that $|sq_1|<|SQ_1|$ and, therefore, $|SQ_1| < |sq_2|$ implies $|sq_1| < |sq_2|$. The only-if part is proved in Lemma \ref{lem:s0<s1<=s0+S0,S0<S1}. 
\end{proof}
The structure of FS-double squares of an unequal 2FS square shown in Theorem \ref{thrm:unequal2FS} fits into the definition of $\delta$ mates defined in \cite{deza15}:
\begin{definition}[$\delta$ mate \cite{deza15}]\label{def:delta}
	An FS-double square $SQ_k^2=((y_1y_2)^{q_1}(y_1)(y_1y_2)^2)^2$ is $\delta$ mate of an FS-double $SQ_1^2$ if it satisfies the next conditions:
	\begin{enumerate}
		\item $1<k<(p_1-1)|x_1x_2|+|lcp(x_1x_2,x_2x_1)|$,
		\item $|sq_k|>|SQ_1|$, and
		\item $s.x_2x_1(x_1x_2)^{p_1+p_2-1}x_1$ is prefix of $sq_k$ for some non-empty suffix $s$ of $x_1$, or
		\item $s(x_1x_2)^ix_1(x_1x_2)^{p_1+p_2-1}x_1$ is a non-trivial prefix of $sq_k$ for some suffix $s$ of $x_1x_2$ and some $i\geq 1$.
	\end{enumerate}
\end{definition}
Given an unequal 2FS square, the first two criteria are satisfied by an FS-double square starting at location $2$. We can write $SQ_2^2=x_1'x_2(x_1x_2)^{p_1-1}(x_1)(x_1x_2)^{p_1+p_2}\cdots$ where it satisfies the third criterion of Definition \ref{def:delta} if $p_1=1$. Otherwise, $SQ_2^2$ begins with the prefix mentioned in the last criterion of the same definition. 
\begin{theorem}[2FS Square]\label{thrm:2FS}
	A 2FS square is exactly one of the following types:
	\begin{enumerate}[(a)]
		\item \textit{Equal 2FS square} with $|sq_1|=|sq_2|$ and $|SQ_1|=|SQ_2|$, or
		\item \textit{Unequal 2FS square} with $|sq_1|<|SQ_1|<|sq_2|<|SQ_2|$.
	\end{enumerate}
\end{theorem}
\begin{proof} The list of possible conditions after removing equivalent cases and those cases that violate trivial conditions, such as $|sq_1| > |SQ_1|$, is given in the second column of Tables \ref{tab:infeasible} and \ref{tab:feasible}. The third column of the Tables gives the lemma number that investigates the corresponding case. Among the thirteen cases, a 2FS square is possible only under two conditions, viz. Case \ref{caseone} and \ref{casetwo} (see Table \ref{tab:feasible}).
	
	\begin{table}[!htb]
		\begin{center}
			\begin{tabular}{r@{\hspace*{10pt}}l@{\hspace*{10pt}}r}
				\hline \bf Case & \hspace*{1cm} \bf Condition & \bf Result \\\hline
				\rownumber & $|sq_1| < |sq_2| < |SQ_1| = |SQ_2|$ & \multirow{4}{*}{Theorem \ref{thrm:equal2FS}}\\
				\rownumber & $|sq_2| < |sq_1| < |SQ_1| = |SQ_2|$ & \\
				\rownumber & $|sq_1| = |sq_2| < |SQ_1| < |SQ_2|$ & \\
				\rownumber & $|sq_1| = |sq_2| < |SQ_2| < |SQ_1|$ & \\
				\hline 
				\rownumber & $|sq_2| < |sq_1| = |SQ_2| < |SQ_1|$ & \multirow{4}{*}{Lemma \ref{lem:s0>s1,S0>S1 and s0>s1,S0<S1}}\\
				\rownumber & $|sq_2| < |sq_1| < |SQ_2| < |SQ_1|$ & \\
				\rownumber & $|sq_2| < |SQ_2| < |sq_1| < |SQ_1|$ & \\
				\rownumber & $|sq_2| < |sq_1| < |SQ_1| < |SQ_2|$ & \\
				\hline 
				\rownumber & $|sq_1| < |sq_2| < |SQ_2| < |SQ_1|$ & Lemma \ref{lem:s1<s2implyS1<S2}\\\hline
				\rownumber & $|sq_1| < |SQ_1| = |sq_2| < |SQ_2|$ & \multirow{2}{*}{Lemma \ref{lem:s0<s1<=s0+S0,S0<S1}}\\
				\rownumber & $|sq_1| < |sq_2| < |SQ_1| < |SQ_2|$ & \\
				\hline
			\end{tabular}
		\end{center}
		\caption{Infeasible relations between lengths of squares for 2FS squares}
		\label{tab:infeasible}
		\vspace{-4mm}
	\end{table}
	
	\begin{table}[!htb]
		\begin{center}
			\begin{tabular}{r@{\hspace*{10pt}}l@{\hspace*{10pt}}c}
				\hline \bf Case & \hspace*{1cm}\bf Condition & \bf Result \\\hline
				\rownumber \label{caseone} & $|sq_1| = |sq_2| < |SQ_1| = |SQ_2|$ & Theorem \ref{thrm:equal2FS} \\
				\rownumber \label{casetwo} & $|sq_1| < |SQ_1| < |sq_2| < |SQ_2|$ & Theorem \ref{thrm:unequal2FS} \\
				\hline
			\end{tabular}
		\end{center}
		\caption{Feasible relations between lengths of squares for 2FS squares}
		\label{tab:feasible}
		\vspace{-8mm}
	\end{table} 
\end{proof}

\begin{corollary}
	The two FS-double squares starting at adjacent locations are either $\alpha$ or $\delta$ mates.
\end{corollary}
A single letter is appended to an FS-double square in order to generate an equal 2FS square.  However, in case of an unequal 2FS square, overall word length depends upon the length of $SQ_2^2$. To find the minimum length of an unequal 2FS square, we first obtain some properties of the square $sq_2^2$ which are used to generate the smallest $SQ_2^2$.
\begin{lemma}
	The relation $|sq_{2}|>|SQ_1|+|sq_1|$ holds for an unequal 2FS square.
\end{lemma}
\begin{proof}(\textit{By contradiction}) 
	Suppose the shorter square, $sq_2^2$, in an unequal 2FS square satisfies the relation $|sq_{2}|\leq|SQ_1|+|sq_1|$.  We have another relation, $|sq_1|<|SQ_1|<|sq_2|<|SQ_2|$, from Theorem \ref{thrm:unequal2FS}. 
	Assume that the square $SQ_1^2$ begins with a letter $a$. If it is appended by the same letter, then the second location begins with the last occurrence of a conjugate of $SQ_1^2$. Lemma \ref{lem:SQ0=SQ1} and Theorem \ref{thrm:equal2FS} show that $SQ_1^2.a$ is an equal 2FS square. Thus, a letter other than $a$, say $b$ appears at the location $2|SQ_1|$. We verify the structures of 2FS squares where the root $sq_2$ satisfies either $|sq_2|=|sq_1|+|SQ_1|$ or $|sq_2|<|sq_1|+|SQ_1|$.
	\begin{description}
		\item [Case1 $|sq_2|=|sq_1|+|SQ_1|$:] The roots of the two FS-double squares follow the alignment as shown in the next equation,
		\begin{align*}
		SQ_1^2=&a(x_1'x_2a)^{p_1}(x_1'a)(x_1'x_2a)^{p_2}(x_1'x_2a)^{p_1}(x_1'a)&&(x_1'x_2a)^{p_2-1}(x_1'x_2\mathbf{b})\cdots\\
		SQ_2^2=&~(y_1y_2)^{q_1}~~~~~~~~~~~~(y_1)&&\underbrace{(y_1y_2)^{q_1}(y_1)}_{(x_1'x_2a)^{p_1}(x_1'a)\cdots}\cdots
		\end{align*}
		If $|lcp(x_1x_2,x_2x_1)|>0$ and the second location starts with the rightmost conjugate of $sq_1$, then to avoid an equal 2FS square, the word $SQ_1^2$ must be appended by any letter except `$a$' and the word structure is given in the first equation of above equation set. Here, the second occurrence of $sq_2$ shows that the word $(x_1'x_2a)^{p_2-1}(x_1'x_2b)$ is a prefix of $(x_1'x_2a)^{p_1}$. We get $b=a$, thus a contradiction.
		\item [Case2 $|SQ_1|<|sq_2|<|sq_1|+|SQ_1|$:] The following arrangement shows different squares in a given 2FS square.
		\begin{align}
		SQ_1^2=&a(x_1'x_2a)^{p_1}(x_1'a)\overbrace{(x_1'x_2a)^{p_2}(x_1'x_2a)^{p_1}}^{(y_1y_2)^{q_1}\cdots}(x_1'a)(x_1'x_2a)^{p_2-1}(x_1'x_2\mathbf{b})\cdots\label{1}\\
		SQ_2^2=&~(y_1y_2)^{q_1}(y_1)~~~~~~~\mathbf{(y_1y_2)^{q_2}(y_1y_2)^{q_1}}(y_1)~~~~~~~(y_1y_2)^{q_2}\label{2}
		\end{align}
		It can be concluded from the marked and highlighted factors in Equation (\ref{1}) and Equation (\ref{2}), respectively, that $y_1y_2$ is a non-primitive word. However, an FS-double square starting in the second location indicates that $y_1y_2$ is a primitive word. This is a contradiction. 
	\end{description}
\end{proof}
\begin{lemma}\label{lem:peoperties of 2FS}
	The following statements hold for an unequal 2FS square with $SQ_1^2=((x_1x_2)^{p_1}(x_1)(x_1x_2)^{p_2})^2$ and $SQ_2^2={((y_1y_2)^{q_1}(y_1)(y_1y_2)^{q_2})}^2$:
	\begin{enumerate}[(a)]
		\item \label{c1} $|sq_2|\geq |SQ_1|+|sq_1|+(p_2-1)(|x_1|+|x_2|)$,
		\item \label{c2} $|y_1y_2|>|x_1x_2|$, and
		\item \label{c3} $|SQ_2|>2|SQ_1|$.
	\end{enumerate}
\end{lemma}
\begin{proof}
	\begin{enumerate}[(a)]
		\item The structure of $SQ_2^2$ in the 2FS square is as follows:
		$$SQ_2^2=a(x_1'x_2a)^{p_1}(x_1'a)(x_1'x_2a)^{p_2}(x_1'x_2a)^{p_1}(x_1'a)(x_1'x_2a)^{p_2-1}(x_1'x_2\mathbf{b})\cdots$$
		The FS-double square $SQ_1^2$ must end with a letter other than `$a$', else the first location begins with only one square or the given 2FS square will be an equal 2FS square. So, the length of $y_1$ must be less than $p_2|x_1x_2|$  $=(p_2-1)|(x_1'x_2a)|+|(x_1'x_2b)|$. We now explore the possible structures of 2FS squares for different $sq_2^2$.
		\newline Case1 : $sq_2^2=(x_1'x_2a)^{p_1}(x_1'a)(x_1'x_2a)^{p_1+p_2}(x_1'a)(x_1'x_2a)^{p_3}$, where $p_3<(p_2-1)$ -- Given 2FS square can be written as follows,
		\begin{align*}
		a\underbrace{(x_1'x_2a)^{p_1}(x_1'a)(x_1'x_2a)^{p_1+p_2}(x_1'a)(x_1'x_2a)^{p_3}}_{P}\underbrace{(x_1'x_2a)^{p_2-p_3-1}(x_1'x_2b) \cdots }_{S}
		\end{align*}
		Note that, in above equation, $P$ is also a prefix of $S$ and $p_2-p_3-1<p_1-1$. This shows that the letters $a$ and $b$ are the same. So, such a structure does not exist.
		\newline Case2: $sq_2=(x_1'x_2a)^{p_1}(x_1'a)(x_1'x_2a)^{p_1+p_2}(x_1'a)(x_1'x_2a)^{p_3}u_1$, where $x_1'x_2a=u_1u_2$ and $x_1'x_2b=u_1u_2'$ such that $u_2\neq u_2'$ -- The 2FS square has a structure,
		\begin{align*}
		a\underbrace{(x_1'x_2a)^{p_1}(x_1'a)(x_1'x_2a)^{p_1+p_2}(x_1'a)(x_1'x_2a)^{p_3}u_1}_{P=sq_1}\underbrace{u_2(x_1'x_2a)^{p_2-p_3-2}(x_1'x_2b) \cdots }_{S}
		\end{align*}
		Again, by noting that $P$ is a prefix of $S$, we have $u_1u_2=u_2u_1$. The relation $p_2-p_3-2<p_1$ implies that the word $u_2'$ is a prefix of $u_1$. This leads to a relation $u_2=u_2'$ which is infeasible.
		\newline So, \ref{c1} holds as the length of a root $sq_2$ satisfies the relation,
		$$|sq_2|\geq |((x_1'x_2a)^{p_1}(x_1'a)(x_1'x_2a)^{p_1+p_2}(x_1'a)(x_1'x_2a)^{p_2-1})|$$
		\item We now find the smallest value of $|y_1y_2|$ to get $|SQ_2|$. Let us first verify the structure of words where $|(y_1y_2)|\leq |(x_1x_2)|$. We assume, $x_1x_2a=u_1u_2'a$, $x_1'x_2b=u_1u_2'b$. We have, $sq_2=(y_1y_2)^{q_1}(y_1)$ and the given 2FS square is,
		\begin{align} \label{eq:eq16}
		a\overbrace{\underbrace{(x_1'x_2a)^{p_1}(x_1'a)(x_1'x_2a)^{p_2}}_{P=(y_1y_2)^{q_3}y_3}.\underbrace{(x_1'x_2a)^{p_1}(x_1'a)(x_1'x_2a)^{p_2-1}(u_1}_{S=y_4(y_3y_4)^{q_4}y_1}.}^{sq_2=(y_1y_2)^{q_1}y_1}u_2'b)\ldots
		\end{align}
		According to Lemma \ref{lem:2FS structure}, $y_1y_2\neq y_2y_1$. Here, $(x_1'x_2a)^{p_1}(x_1'a)(x_1'x_2a)^{p_2}=(y_1y_2)^{q_3}y_3$ and $y_1y_2=y_3y_4$. Since $S$ is a prefix of $P$, $y_3y_4=y_4y_3$. This shows that the word $y_1y_2$ is a non-primitive word leading to a contradiction.
		\item The length of a longer root $SQ_2$ is equals to $|sq_2|+q_2|(y_1y_2)|$. The statements \ref{c1} and \ref{c2} give us $|SQ_2|>2|SQ_1|$.
	\end{enumerate}
\end{proof}
\begin{lemma}\label{lem:no of equal 2FS}
	Let $w$ be a word that starts with an FS-double square $SQ_1^2=((x_1x_2)^{p_1}(x_1)(x_1x_2)^{p_2})^{2}$, for some nonempty words $x_1,x_2$ and integers $p_1,p_2$ where $p_1\geq p_2\geq 1$. If the first $k$ locations in $w$ start with equal length FS-double squares, then
	$$
	k \leq \left\{ 
	\begin{array}{ll}
	\min(lcp(x_1x_2,x_2x_1), |x_1|-1) & \textnormal{ where } p_1=p_2\\
	lcp(x_1x_2,x_2x_1) & \textnormal{ otherwise} 
	\end{array}\right.
	$$
\end{lemma}
\begin{proof}
	As per the definition of an equal 2FS square, all the roots 
	$sq_1,sq_2,\\ \ldots,sq_k$ are of equal length and are conjugates. From Lemma \ref{lem:appendingPrefix}, it is required to append the prefix of $sq_1$ to the square $sq_1^2$ to obtain the conjugates of $sq_1$ at successive locations. 
	
	For $p_1=p_2=p$, the square $sq_1^2$ in $SQ_1^2$ is followed by $(x_2x_1)^{p}(x_1x_2)^{p}$ (see Equation (\ref{eq:lcp})). Since $x_1x_2\neq x_2x_1$, the  number of conjugates of the square $sq_1^2$ is $|lcp(x_1x_2,x_2x_1)|$. However, $sq_1^2$ occurs at location $|SQ_1|$ if $SQ_1^2$ is appended by its own prefix of length $\geq |x_1|$. This is shown in Equation (\ref{eq:lcp2}).
	\begin{align}
	SQ_1^2&=\underbrace{(x_1x_2)^{p}(x_1)(x_1x_2)^{p}x_1}_{sq_1^2}(x_2x_1)^{p}(x_1x_2)^{p}\label{eq:lcp}\\
	SQ_1^2.x_1&=(x_1x_2)^{p}(x_1)(x_1x_2)^{p}\boxed{(x_1x_2)^{p}(x_1)(x_1x_2)^{p}(x_1)}\implies{sq_1^2}\label{eq:lcp2}
	\end{align}
	Thus, the number of conjugates of $sq_1^2$ that are adjacent to each other is less than or equal to $|x_1|-1$.
	Similarly, $SQ_1,SQ_2,\cdots,SQ_k$ are conjugates. Since $x_1x_2\neq x_2x_1$ the longest prefix is of $|lcp(x_1x_2,x_2x_1)|$ length. From Theorem \ref{thrm:equal2FS}, the number of consecutive FS-double squares depends upon the number of conjugates of both $SQ_1^2$ and $sq_1^2$. So, $k\leq \min(lcp(x_1x_2,x_2x_1), |x_1|-1)$. For $p_1>p_2$, $k\leq lcp(x_1x_2,x_2x_1)$ as appending $x_1$ to $SQ_1^2$ does not reduce the value of $s_1$.
\end{proof}

\section{Longest Sequence of 2's}\label{sec:longest sequence of FS}
We have shown that a word that starts with consecutive FS-double squares, any two consecutive squares follows the structure of either equal or unequal 2FS square. The $s_i$ sequence of such a word has a chain of $2$s in the beginning. A word $w$ has a sequence of 2's if $s_i(w)=s_{i+1}(w)=\cdots =s_j(w)=2$ where the integers $i,j$ satisfy $1\leq i<j\leq|w|$. It is possible to extend an FS-double square to get an arbitrarily long sequence of 2's. One way to achieve this is described in Lemma \ref{lem:no of equal 2FS}, where an FS-double square is appended by its prefix. In this case, all the consecutive FS-double squares in the beginning of a word are conjugates of each other. The number of such FS-double squares is finite, and the length of a sequence of 2's is limited. However, it is always possible to introduce an unequal 2FS square to increase the length of the sequence of 2's.
Thus, we can extend an FS-double square to get a sequence of 2's of any desired length by introducing a new equal or unequal 2FS square. A single letter is added to an FS-double square to introduce a new equal 2FS square, whereas an FS-double square is appended by many letters to get a new unequal 2FS square. Let us see some equal and unequal 2FS squares. Next is an example of equal 2FS square and its $s_i$ sequence. Here, $SQ_1^2=((aba)^1.(ab).(aba)^1)^2$ and $SQ_2^2=((baa)^1.(ba).(baa)^1)^2$.

\begingroup
\setlength{\tabcolsep}{1pt}
\renewcommand{\arraystretch}{.5}
\begin{center}
	\begin{tabular}{ccccccccccccccccccc}
		$w$ & = & $a$ & $b$ & $a$ & $a$ & $b$ & $a$ & $b$ & $a$ & $a$ & $b$ & $a$ & $a$ & $b$ & $a$ & $b$ & $a$ & $a$\\
		$s_i(w)$ & = & 2 & 2 & 0 & 0 & 0 & 0 & 1 & 1 & 1 & 0 & 0 & 1 & 1 & 0 & 0 & 1 & 0\\
	\end{tabular}
\end{center}
\endgroup
The word, $w$, if continued to be extended further with the prefix of $SQ_1^2$, then $sq_1^2$ repeats after the first location and the value of $s_1$ reduces to one. In such words, it is necessary to introduce unequal 2FS square to continue the sequence of 2's further. Unlike equal 2FS squares, the structures of unequal 2FS squares vary, and there are different ways to extend an FS-double square to get an unequal 2FS square. To elaborate this, we extend an FS-double square in two different ways to get two different unequal 2FS squares. Let $SQ_1^2=aabaaabaabaaab$ be an FS-double square which is extended to get two unequal 2FS squares $w_1$ and $w_2$, where
\begin{align}
w_1=a&((abaaabaabaaabb)(ab)(abaaabaabaaabb))^2\\
w_2=a&((abaaabaabaaabb)(abaaabaabaaabbab)(abaaabaabaaabb))^2
\end{align}
and the respective $s_i$ sequences are,
\newline \begingroup
\setlength{\tabcolsep}{1pt}
\renewcommand{\arraystretch}{.5}
\begin{center}
	\scalebox{0.95}{%
		\begin{tabular}{ccccccccccccccccccccccccccccccccccc}
			$w_1$ & = & a & a & b & a & a & a & b & a & a & b & a & a & a & b & b & a & b & a & b & a & a & a & b & a & a & b & a & a & a & b & b\\
			& & a & b & a & a & a & b & a & a & b & a & a & a & b & b & a & b & a & b & a & a & a & b & a & a & b & a & a & a & b & b &  \\
			$s_i(w_1)$ & = & 2 & 2 & 1 & 0 & 0 & 0 & 0 & 0 & 0 & 0 & 0 & 0 & 0 & 0 & 0 & 0 & 1 & 1 & 1 & 1 & 1 & 0 & 0 & 0 & 0 & 0 & 0 & 0 & 0 & 0 & 0 \\
			& & 0 & 0 & 0 & 0 & 0 & 0 & 0 & 0 & 0 & 0 & 0 & 0 & 0 & 0 & 1 & 1 & 1 & 0 & 0 & 1 & 1 & 1 & 0 & 0 & 0 & 0 & 1 & 0 & 1 & 0 & \\
	\end{tabular}}
\end{center}
\endgroup
\begingroup
\setlength{\tabcolsep}{1pt}
\renewcommand{\arraystretch}{.5}
\begin{center}
	\scalebox{0.95}{%
		\begin{tabular}{ccccccccccccccccccccccccccccccccccccccccccccccc}
			$w_2$ & = & a & a & b & a & a & a & b & a & a & b & a & a & a & b & b & a & b & a & a & a & b & a & a & b & a & a & a & b & b & a & b & a & b & a & a & a & b & a & a & b & a & a & a & b & b\\
			& & a & b & a & a & a & b & a & a & b & a & a & a & b & b & a & b & a & a & a & b & a & a & b & a & a & a & b & b & a & b & a & b & a & a & a & b & a & a & b & a & a & a & b & b\\
			$s_i(w_2)$ & = & 2 & 2 & 1 & 0 & 0 & 0 & 0 & 0 & 0 & 0 & 0 & 0 & 0 & 0 & 1 & 1 & 1 & 0 & 0 & 0 & 0 & 0 & 0 & 0 & 0 & 0 & 0 & 0 & 0 & 0 & 0 & 0 & 0 & 0 & 0 & 1 & 1 & 1 & 1 & 1 & 1 & 1 & 1 & 1 & 1\\
			& & 1 & 1 & 1 & 1 & 0 & 0 & 0 & 0 & 0 & 0 & 0 & 0 & 0 & 0 & 0 & 0 & 0 & 0 & 0 & 0 & 0 & 0 & 0 & 0 & 0 & 0 & 0 & 0 & 1 & 1 & 1 & 0 & 0 & 1 & 1 & 1 & 0 & 0 & 0 & 0 & 1 & 0 & 1 & 0\\
	\end{tabular}}
\end{center}
\endgroup

A word can be extended to get an unequal 2FS square at any location. Moreover, it is possible to yield an unequal 2FS square at a particular location $l$ such that it does not affect the $s_i$ value of another location $m$ where $0<m<l$. This new 2FS square almost doubles the overall word length though. So, we investigate the relationship between the length of the longest sequence of 2's and the word length. It is evident that in a word $w$, the ratio of the longest sequence of 2's to its word length is higher for equal 2FS squares. Following lemma computes the ratio for the sequence of 2's such that any two consecutive FS-double squares in the sequence follow the structure of an equal 2FS square.
\begin{lemma}[Longest sequence of 2's with Equal 2FS Squares]\label{lem:max_equal}
	Let $T$ be the longest sequence of consecutive FS-double squares in $w$ such that any two consecutive FS-double squares in $T$ are conjugates. Then, $\frac{|T|}{|w|}\leq \frac{1}{7}$.
\end{lemma}
\begin{proof}
	Assume that the first FS-double square in $T$ is $((x_1x_2)^{p_1}(x_1)(x_1x_2)^{p_2})^2$ where $x_1,x_2\in \varSigma^+$ and integers $p_1,p_2$ satisfy the relation $p_1\geq p_2\geq 0$. From Lemma \ref{lem: max equal 2FS}, the length of $T$ depends on the values of $p_1$ and $p_2$. For $p_1=p_2$, The highest value of the ratio $|T|/|w|$ is,
	\begin{align*}
	\frac{|T|}{|w|}&=\frac{\min(|lcp(x_1x_2,x_2x_1)|,(|x_1|-1))}{2((p_1+p_2+1)|x_1|+(p_1+p_2)|x_2|)+|lcp(x_1x_2,x_2x_1)|}\\
	%&=\frac{\min((|x_1|+|x_2|-2),(|x_1|-1))}{2((p_1+p_2+1)|x_1|+(p_1+p_2)|x_2|)+|lcp(x_1x_2,x_2x_1)|}\\
	%&=\frac{|x_1|-1}{2((p_1+p_2+1)|x_1|+(p_1+p_2)|x_2|)+(|x_1|-1)}\\
	&=\frac{|x_1|-1}{(2p_1+2p_2+3)|x_1|+2(p_1+p_2)|x_2|-1} \leq \frac{1}{7}
	\end{align*}
	
	The highest value of the ratio $|T|/|w|$ for $p_1>p_2$ is,
	\begin{align*}
	\frac{|T|}{|w|}&= \frac{|lcp(x_1x_2,x_2x_1)|}{2((2+1+1)|x_1|+(2+1)|x_2|)+|lcp(x_1x_2,x_2,x_1)|}\\
	&=\frac{|x_1|+|x_2|-2}{8|x_1|+6|x_2|)+|x_1|+|x_2|-2}=\frac{|x_1|+|x_2|-2}{9|x_1|+7|x_2|)-2} \leq \frac{1}{7}
	\end{align*}
	Thus, the highest value of the ratio $\frac{|T|}{|SQ_i|}$ is $\frac{1}{7}$.
\end{proof}

The sequence of consecutive FS-double squares referred in Lemma \ref{lem:max_equal} can be further extended by adding a new unequal 2FS square. So, another way to generate a long sequence of $2$'s is to start with an FS-double square and extend it to add all possible conjugates of the square. At this point, we can append the word to generate an unequal 2FS square so that the sequence of $2$'s continues to grow. Thus, a sequence increases either with equal or unequal 2FS square. The length of such a sequence in a word with respect to the word length is computed in the following lemma.

\begin{lemma}[Longest Sequence of 2's with Equal and Unequal 2FS squares]\label{lem:Equal and unequal}
	Let $T$ be the longest sequence of FS-double squares in a word $w$ that obtained using the best combination of equal and unequal 2FS squares. Then, $\frac{|T|}{|w|}\leq\frac{1}{15}$.
\end{lemma}
\begin{proof} The length of $T$ can be increased by adding a new 2FS square.
	Every new equal 2FS square increments the value of both $|T|$ and $|w|$ by $1$. This improves the value of $\frac{|T|}{|w|}$. However, it is not always possible to introduce an equal 2FS square (see Lemma \ref{lem: max equal 2FS}) and, therefore, an unequal 2FS square is required to get a longer $T$. Unlike an equal 2FS square, a new unequal 2FS square decreases the value of $\frac{|T|}{|w|}$. To understand this, suppose an unequal 2FS square begins at location one where $SQ_1^2$ and $SQ_2^2$ are two consecutive FS-double squares with shorter squares $sq_1^2$ and $sq_2^2$ respectively. Lemma \ref{lem:peoperties of 2FS} gives the relation $|SQ_2|>2|SQ_1|$. Accordingly, $SQ_1^2$ is appended 
	by a word containing at least $2|SQ_1|$ letters to make $s_2=2$. Thus, the value of $\frac{|T|}{|w|}$ decreases significantly after introducing an unequal 2FS square. So, we can obtain the best ratio from the word that has maximum equal 2FS squares and some unequal 2FS squares. 
	
	Given a word with $s_1=s_2=\cdots=s_i=2$ such that the location $i$ starts with an FS-double square $SQ_i^2$. From Lemma \ref{lem:max_equal}, the sequence of 2's can be extended to get at most $\frac{|SQ_i|}{7}$ new equal 2FS squares. An unequal 2FS square must be introduced to continue the sequence of 2's further.
	The ratio $\frac{|T|}{|w|}$ for the smallest FS-double square $w=(abaab)^2$ is $\frac{1}{10}$. We use the above method to extend $w$. The respective $\frac{|T|}{|w|}$ obtained after introducing a new unequal 2FS square results into the following sequence.
	\begin{align*}
	&\frac{1+\frac{10}{7}+1}{10+\frac{10}{7}+(10+10)}, \frac{1+\frac{10}{7}+1+\frac{20}{7}+1}{10+20+\frac{10}{7}+\frac{20}{7}+(20+20)},\\[10pt]
	&\frac{1+\frac{10}{7}+1+\frac{20}{7}+1+\frac{40}{7}+1}{10+20+40+\frac{10}{7}+\frac{20}{7}+\frac{40}{7}+(40+40)},\cdots
	\end{align*}
	The $n^{th}$ term of the above sequence is
	$$\frac{(n+1)+\frac{1}{7}*(10+20+40+\cdots+10*2^n)}{\frac{1}{7}*(10+20+40+\cdots+10*2^n)+(10+20+40+\cdots+10*2^{n+1})}$$
	The highest possible value for $\frac{|T|}{|w|}$ is obtained by simplifying the $n^{th}$ term as follows.
	\begin{align*}
	\frac{|T|}{|w|}&=\frac{(10+20+\cdots+2^n*10)+7n+7}{(10+20+\cdots+2^n*10)+7(10+20+\cdots+2^{n+1}*10)}\\[10pt]
	&=\frac{10*2^n-10+7n+7}{10*2^n-10+14*10*2^n-70}
	=\frac{(10*2^n)+7n-3}{15*(10*2^n)-8}\leq \frac{1}{15}
	\end{align*}
\end{proof}

\begin{theorem}
	If $T$ is the longest sequence of $s_i=2'$s in a word $w$, then $|T|<\frac{|w|}{7}$.
\end{theorem}
\begin{proof} The computation in Lemmas \ref{lem:max_equal} and \ref {lem:Equal and unequal} shows that the best value of $\frac{|T|}{|w|}$ where $T$ contains either equal length FS-double squares or a combination of consecutive FS-double squares that can be equal or unequal 2FS squares. The best value is obtained in the former case, that is, $\frac{1}{7}$. We compare this value with the length of the only remaining possible sequence of $T$ where every two consecutive FS-double squares follow the structure of an unequal 2FS square. 
	
	Suppose $SQ_1^2$ and $SQ_2^2$ results into an unequal 2FS square at the beginning of a word. Then, $|SQ_2|\geq 2|SQ_1|$ (see Lemma \ref{lem:peoperties of 2FS}). Thus, to introduce a new unequal 2FS square at location $i$, it is required to append at least $2|SQ_i|$ letters to the FS-double square $SQ_i^2$. Allowing only unequal 2FS squares in $T$, we compute the ratio of the length of the longest sequence of 2's in a word to its word length as follows.
	\begin{align*}
	&\frac{1}{10}, \frac{1+1}{10+20}, \frac{1+1+1}{10+20+40}, \cdots, \frac{n}{10*(2^{n}-1)}, \cdots
	\end{align*}
	The value of the ratio decreases as $n$ increases and the ratio has the maximum value of $\frac{1}{15}$ for $n=2$. We ignore the value with $n=1$ as the sequence will have only one FS-double square. Therefore, $|T|\leq \frac{|w|}{7}$.
\end{proof}

\section{Conclusion}\label{sec:conclusion}
We have investigated the sequence of FS-double squares and introduced the term 2FS square for two consecutive FS-double squares. The 2FS squares are characterized into two types viz.\ equal and unequal 2FS squares. In equal 2FS square, a letter is added to the existing FS-double square to obtain a new FS-double square. In contrast, an FS-double square is appended by a word of its own length to yield an unequal 2FS square. It is shown that in spite of the long suffixes, an unequal 2FS square can be used to introduce any number of consecutive FS-double squares in a word. We have compared the maximum number of successive FS-double squares in a word with its length. The equal length FS-double squares produce a word, $w$, with the longest sequence of $s_i=2'$s, $T$, and the ratio $\frac{|T|}{|w|}$ converges to $\frac{1}{7}$.

The number of distinct squares in a word depends on the non-zero $s_i$ indexes. A possible direction to explore the square-conjecture is to study the structure of words with a large number of indexes with $s_i = 2$. We believe that the structure of a 2FS square can be used to explore the square-conjecture.

\bibliographystyle{splncs04}
\bibliography{references}

\end{document}